\documentclass[a4paper,11pt]{article}
\usepackage[utf8]{inputenc}
\usepackage[T1]{fontenc}
\usepackage{textcomp}
\usepackage{fancyhdr}
\usepackage{fullpage}
\usepackage{lmodern}
\usepackage{amsmath,amssymb,bm,bbm}
\usepackage{algorithm}
\usepackage{algorithmic}
\usepackage{float}
\usepackage{graphicx}
\usepackage{extarrows}
\usepackage{caption}
\usepackage{graphicx, subfig}
\usepackage{titling}
\usepackage{yhmath}
\usepackage{mathrsfs}
\usepackage{cite}
\usepackage{footnote}
\usepackage{amsthm}
\usepackage{color}
\usepackage{slashed}
\usepackage{verbatim}
\usepackage[colorlinks=true,linkcolor=red,citecolor=blue]{hyperref}
\usepackage{amsfonts,euscript}
\usepackage{latexsym, multicol}
\usepackage{epstopdf}
\usepackage{psfrag}
\usepackage{mathrsfs}
\usepackage{xcolor}
\usepackage{comment}
\usepackage{authblk}
\usepackage{indentfirst}
\usepackage{lipsum}
\DeclareFontFamily{U}{mathx}{\hyphenchar\font45}
\DeclareFontShape{U}{mathx}{m}{n}{
<5> <6> <7> <8> <9> <10>
<10.95> <12> <14.4> <17.28> <20.74> <24.88>
mathx10}{}

\DeclareSymbolFont{mathx}{U}{mathx}{m}{n}
\DeclareFontSubstitution{U}{mathx}{m}{n}
\DeclareMathAccent{\widecheck}{0}{mathx}{"71}
\numberwithin{equation}{section}
\allowdisplaybreaks[3]

\def\etabf{\bm{\eta}}

\def\RRR{\mathbb{R}}
\def\SSS{\mathbb{S}}
\def\xja{\langle x\rangle}

\def\NNN{\mathbb{N}}
\def\ges{\gtrsim}
\renewcommand{\c}{\cdot}

\newcommand{\DD}{\mathscr{D}}
\newcommand{\pa}{\partial}
\newcommand{\pr}{\pa}
\newcommand{\unu}{{\underline{u}}}
\newcommand{\unc}{{\underline{C}}}

\newcommand{\bg}{\mathbf{g}}

\newcommand{\M}{\mathcal{M}}

\newcommand{\Ric}{{\ric}}

\newcommand{\g}{\bg}
\newcommand{\R}{{\mathbf{R}}}

\renewcommand{\div}{\sdiv}

\def\ub{\unu}

\def\Cb{\unc}
\def\om{\omega}

\def\Om{\Omega}

\def\Si{\Sigma}
\def\si{\sigma}

\def\Ga{\Gamma}

\def\a{\alpha}
\def\b{\beta}

\def\de{\delta}
\def\De{\Delta}
\def\nab{\nabla}
\def\ov{\overline}

\def\ep{\varepsilon}
\def\les{\lesssim}

\def\th{\theta}

\DeclareMathOperator{\supp}{supp}

\DeclareMathOperator{\tr}{tr}
\DeclareMathOperator{\sdiv}{div}

\def\ric{\mathbf{Ric}}
\newtheorem{thm}{Theorem}[section]
\newtheorem{prop}[thm]{Proposition}
\newtheorem{lem}[thm]{Lemma}
\newtheorem{cor}[thm]{Corollary}
\newtheorem{rk}[thm]{Remark}
\newtheorem{df}[thm]{Definition}

\title{Vacuum initial data with minimal decay and borderline decay}
\author{Dawei Shen and Jingbo Wan}
\begin{document}
\maketitle
\begin{abstract}
In this note, we show that the conical solution-operator method of Mao-Tao \cite{MaoTao} applies to a simple construction of \emph{vacuum} asymptotically flat initial data at minimal and borderline decay thresholds, corresponding to the global and exterior stability of Minkowski spacetime proved by the first named author in \cite{Shen23,Shen24}.
\end{abstract}
% \tableofcontents
\section{Introduction}
The Einstein vacuum equations read as follows:
\begin{equation*}
    \Ric(\M,\g)=0,
\end{equation*}
where $(\M,\g)$ is a Lorentzian manifold and $\Ric(\M,\g)$ is the Ricci curvature. When we restrict to a spacelike hypersurface, we get the following constraint equations:
\begin{align}
    \begin{split}\label{MT1.1}
    R+(\tr k)^2-|k|^2&=0,\\
    \div(k-(\tr k) g)&=0.
    \end{split}
\end{align}
We begin by introducing the notion of \emph{$(s,q)$--asymptotically flat initial data of size $\ep_0$}; see also \cite[Definition~1.2]{Shen23, Shen24} for closely related definitions.\footnote{The notion of $(s,q)$--asymptotically flat initial data used here differs slightly from that of \cite{Shen23, Shen24}. In particular, we explicitly introduce the size parameter $\ep_0$ and replace the little--$o$ decay assumptions in \cite{Shen23, Shen24} by big--$O$ bounds.}
\begin{df}\label{defasym}
Given $s\geq 1$, $q\in\NNN$ and let $0<\ep_0\ll 1$, we say that a \emph{vacuum} initial data set $(\Si_0,g,k)$ is $(s,q)$--asymptotically flat with size $\ep_0$ if it solves \eqref{MT1.1} and the following properties hold:
\begin{enumerate}
    \item There exists a coordinate system $(x^1,x^2,x^3)$ defined outside a sufficiently large compact set such that:\footnote{The notation $f =O(1)$ means that $f$ is bounded when $r\to\infty$. Moreover, for any $m\in\mathbb{R}$ and $q\in\mathbb{N}$, we write $f=O_q(r^{-m})$ if $\pr^\a f=O(r^{-m-|\a|})$ for all $|\a|\leq q$.}
    \begin{align*}
    g&=\left\{
    \begin{aligned}
    \de+O_{q+1}\left(\frac{\ep_0}{\xja^\frac{s-1}{2}}\right),\qquad k=O_q\left(\frac{\ep_0}{\xja^\frac{s+1}{2}}\right)\qquad&\mbox{ if }\; s\in[1,3),\\
    g_m+O_{q+1}\left(\frac{\ep_0}{\xja^\frac{s-1}{2}}\right),\qquad k=O_q\left(\frac{\ep_0}{\xja^\frac{s+1}{2}}\right)\qquad&\mbox{ if }\; s\geq 3.
    \end{aligned}\right.
    \end{align*}
where we denote $\xja:=\sqrt{1+|x|^2}$ and $g_m$ denotes a $3$--dimensional Schwarzschild metric with mass $m>0$ sufficiently small.
    \item The following $b$--Sobolev bounds hold:\footnote{See Definition \ref{bSobolev} for a precise definition of the weighted $b$--Sobolev spaces $H_b^{k,\de}$.}
    \begin{align}
    \begin{split}\label{Hbflux}
    \|(g-\de,k)\|_{H_b^{q+1,\frac{s-4}{2}}(\Si_0)\times H_b^{q,\frac{s-2}{2}}(\Si_0)}&\leq\ep_0 \qquad\mbox{ if }\; s\in[1,3), \\
    \|(g-g_m,k)\|_{H_b^{q+1,\frac{s-4}{2}}(\Si_0)\times H_b^{q,\frac{s-2}{2}}(\Si_0)}&\leq\ep_0\qquad\mbox{ if }\; s\geq 3.
    \end{split}
    \end{align}
\end{enumerate}
\end{df}
In \cite{Shen23}, the first named author established the global nonlinear stability of Minkowski space for $(s,q)$--asymptotically flat vacuum initial data with $s>1$ and $q\ge 2$, while \cite{Shen24} proved the exterior stability of Minkowski space under the borderline decay assumption $s=1$, in particular for $(1,3)$--asymptotically flat initial data. Fang-Szeftel-Touati \cite{FST} constructed $(s,q)$--asymptotically flat vacuum initial data for $s>1$ and $q\in\NNN$ sufficiently large by the conformal method.

An open problem remaining concerns the existence of initial asymptotically flat vacuum data with borderlines $(1,q)$ for large values of $q$, since the conformal construction in \cite{FST} does not extend to the case $s=1$. The main result of the present paper, Theorem \ref{s=1initial}, resolves this issue by constructing, for any $s\in[1,3)$ and any $q\in\NNN$, a nontrivial $(s,q)$--asymptotically flat vacuum initial data set. Moreover, in the borderline case $s=1$, the initial data constructed are not $(s',q)$–asymptotically flat in the sense of Definition 1.1 for all $s'>1$ and $q\in\NNN$. In particular, this produces initial data that satisfy the hypotheses of \cite{Shen24} but not those of \cite{Shen23}.

The proof of our main result, Theorem \ref{s=1initial}, is modeled on the conic localization program for vacuum constraint equations, in which nontrivial solutions are constructed with support in a prescribed cone, with the goal of retaining a sharp asymptotic decay. This direction was initiated by Carlotto-Schoen \cite{CS}, who achieved conic localization at the expense of a decay loss in the transition region. Carlotto later conjectured that such localization should be possible without loss of decay \cite[Open Problem~3.18]{Car21}. This conjecture was resolved by Aretakis-Czimek-Rodnianski \cite{ACR} by characteristic gluing.

Mao-Tao \cite{MaoTao} later returned to the conic setting and developed a spacelike approach based on explicit solution operators for linearized constraint equations with prescribed conic support; see also \cite{MOT} for solution operators with annulus support. In the present work, we adapt the Mao-Tao framework to construct borderline $(1,q)$--asymptotically flat initial data, thus proving Theorem \ref{s=1initial}. It is also worthwhile to mention the recent elliptic-transport construction of spacelike initial data adapted to a spherical foliation due to Chen-Klainerman \cite{ChenKlainerman}.
\subsection{Stability of Minkowski spacetime}
\subsubsection{Previous works on the stability of Minkowski spacetime}
In 1993, Christodoulou-Klainerman \cite{Ch-Kl} established the global nonlinear stability of Minkowski spacetime for the Einstein vacuum equations. In 2003, Klainerman-Nicol\`o \cite{Kl-Ni} proved the stability of Minkowski space in the exterior of an outgoing null cone by employing a double null foliation framework, under the same assumptions on the initial data as in \cite{Ch-Kl}. Moreover, under stronger assumptions on asymptotic decay and regularity than those imposed in \cite{Ch-Kl, Kl-Ni}, Klainerman-Nicol\`o \cite{knpeeling} showed that asymptotically flat initial data give rise to solutions of the Einstein vacuum equations exhibiting strong peeling properties.

In 2007, Bieri \cite{Bieri} provided an alternative proof of the global stability of Minkowski space, which required fewer derivatives and a reduced collection of vector fields compared to the approach in \cite{Ch-Kl}. Subsequently, Lindblad-Rodnianski \cite{lr1, lr2} introduced a proof based on wave coordinates, showing that in this gauge the Einstein equations satisfy the so-called weak null structure. Using generalized wave coordinates, Huneau \cite{huneau} proved the nonlinear stability of Minkowski spacetime in the presence of a translation Killing field. Using Melrose's $b$--analysis, Hintz-Vasy \cite{hv} obtained another proof of the stability of Minkowski space. In a different direction, Graf \cite{graf} established the global nonlinear stability of Minkowski space in the context of the spacelike-characteristic Cauchy problem for Einstein vacuum equations, which, when combined with \cite{Kl-Ni}, allows one to recover the result of \cite{Ch-Kl}.

Within the framework developed in \cite{Kl-Ni} and using the $r^p$--weighted energy estimates introduced by Dafermos-Rodnianski \cite{Da-Ro}, the first named author \cite{Shen22} reproved the stability of Minkowski space in exterior regions. More recently, Hintz \cite{Hintz} established exterior stability results using the analytic framework of \cite{hv}. Moreover, using $r^p$--weighted estimates, the first named author \cite{Shen23} extended the result of Bieri \cite{Bieri} to initial data satisfying minimal decay assumptions. Subsequently, in \cite{Shen24}, the first named author showed that the exterior stability of Minkowski space persists under decay assumptions that are borderline relative to those considered in \cite{Shen23}.

In addition to the vacuum case, there is extensive work on the stability of Einstein equations coupled to matter fields; see the introductions of \cite{Shen22, Shen23} and the references therein.
\subsubsection{Main results of \texorpdfstring{\cite{Shen23,Shen24}}{}}
We now recall the main results in \cite{Shen23,Shen24}.
\begin{thm}[Global stability of Minkowski with minimal decay \cite{Shen23}]\label{thm:Shen23}
Fix $s>1$ and $0<\ep_0\ll 1$. Let $(\Si_0,g,k)$ be a $(s,2)$--asymptotically flat \emph{vacuum} initial data with size $\ep_0$, in the sense of Definition \ref{defasym}. Then, $(\Si_0,g,k)$ has a unique, globally hyperbolic, smooth, geodesically complete development under Einstein vacuum equations. Moreover, this development is globally asymptotically flat, i.e. the Riemann curvature tensor tends to zero along any future directed causal or spacelike geodesic.
\end{thm}
\begin{thm}[Exterior stability of Minkowski with borderline decay \cite{Shen24}]\label{thm:Shen24}
Let $(\Si_0,g,k)$ be a $(1,3)$--asymptotically flat \emph{vacuum} initial data with size $\ep_0$, in the sense of Definition \ref{defasym}. Let $K\subset\Si_0$ be a compact set such that $\Si_0\setminus K$ is diffeomorphic to $\RRR^3\setminus\ov{B}_1$. Then, the stability of Minkowski holds in the exterior of an outgoing null cone. More precisely, there exists a unique future development $(\M,\g)$ of $\Si_0\setminus K$ with the following properties:
\begin{itemize}
\item $(\M,\bg)$ can be foliated by a double null foliation $(C_u,\Cb_\ub)$ whose outgoing leaves $C_u$ are complete.
\item We have detailed control of all the quantities associated with the double null foliations of the spacetime.
\end{itemize}
\end{thm}
We briefly explain the main ideas of the proofs in \cite{Shen23, Shen24}, in order to clarify the motivation for introducing the notion of $(s,q)$--asymptotically flat initial data.

Let $(\M,\g)$ denote the future development of $(s,q)$--asymptotically flat initial data. For $s\in[1,3)$, on spheres of area radius $r$, the following bootstrap bounds represent the optimal decay rates that one can expect:
\begin{align*}
\g-\etabf = O\!\left(\frac{\ep}{r^{\frac{s-1}{2}}}\right),\qquad
\Ga = O\!\left(\frac{\ep}{r^{\frac{s+1}{2}}}\right),\qquad
\R = O\!\left(\frac{\ep}{r^{\frac{s+3}{2}}}\right),
\end{align*}
where $\etabf$ denotes the Minkowski metric, $\Ga$ the connection coefficients, $\R$ the Weyl curvature components, and $\ep:=\ep_0^{\frac{2}{3}}$ denotes the size of the bootstrap. Moreover, from \eqref{Hbflux} we obtain the weighted Sobolev bound
\begin{align*}
\|\R\|_{H_b^{q-1,\frac{s}{2}}(\Si_0)}\les\ep_0.
\end{align*}
The Weyl curvature satisfies the schematic Bianchi equations
\begin{align}
\nab_t \R = \nab \R + \Ga \c \R, \label{Bianchi1}
\end{align}
where $\nab_t$ denotes the time derivative, while $\nab$ denotes derivatives tangent to the spheres. Applying the $r^p$--weighted estimates of Dafermos-Rodnianski \cite{Da-Ro}, one obtains schematically
\begin{align}\label{rpweighted}
\int_{\Si_t} r^p |\R|^2
\lesssim
\int_{\Si_0} r^p |\R|^2
+ \int_{V_t} r^p |\R \c \Ga \c \R|,
\end{align}
where $t$ is a time function and $V_t$ denotes the spacetime region between $\Si_0$ and $\Si_t$. We refer to {\cite[Section 5] {Shen23}} for further details on this argument.

When $s\in(1,3)$, taking $p=s$ in \eqref{rpweighted}, the nonlinear term can be estimated as
\begin{align*}
\int_{V_t} r^s |\R \c \Ga \c \R|\les\int_{V_t}r^s\frac{\ep}{r^{\frac{s+3}{2}}}\frac{\ep}{r^{\frac{s+1}{2}}}\frac{\ep}{r^{\frac{s+3}{2}}}\les\int_{V_t} \frac{\ep^3}{r^{\frac{s+7}{2}}}\les\ep_0^2,
\end{align*}
which is integrable in spacetime. Consequently, the nonlinear terms decay faster than the linear ones, and the bootstrap argument can be closed at the level of the $r^s$--weighted energy flux.

In contrast, in the borderline case $s=1$, the same estimate yields
\begin{align*}
\int_{V_t} \frac{\ep^3}{r^4} = \infty,
\end{align*}
since $V_t$ is a four-dimensional spacetime region. To recover integrability, one must therefore take $p=1-2\de$ with $0<\de\ll1$ in \eqref{rpweighted}. This leads to energy estimates that yield only weaker decay rates:
\begin{align*}
\Ga = O\!\left(\frac{\ep_0}{r^{1-\de}}\right),\qquad
\R = O\!\left(\frac{\ep_0}{r^{2-\de}}\right).
\end{align*}
To recover the lost factor $r^\de$ and obtain the optimal decay, one must additionally exploit the transport equations in the incoming null direction. This mechanism explains why $(s,2)$--asymptotically flat initial data are considered in \cite{Shen23} for $s>1$, whereas $(1,3)$--asymptotically flat initial data are required in \cite{Shen24}. We refer to {\cite[Section 6] {Shen24}}  for a detailed discussion of the recovery of the $r^\de$--loss via transport equations.
\subsection{Statements of the main theorem}

A comparison between Theorems \ref{thm:Shen23} and \ref{thm:Shen24} highlights a natural gap between the corresponding assumptions on the initial data. More precisely, one may ask whether there exists an initial data set $(\Si_0,g,k)$ which is $(1,3)$--asymptotically flat with size $\ep_0$, but fails to be $(s,2)$--asymptotically flat for any $s>1$. Such initial data would satisfy the assumptions of Theorem \ref{thm:Shen24} while violating those of Theorem \ref{thm:Shen23}, and would therefore give rise to a spacetime for which exterior stability holds in the future, whereas the question of global stability remains open.

To formulate our main result, we first introduce the notion of conical support.
\begin{df}\label{conedf}
For $\om\in\SSS^2$ and $\th\in(0,\frac{\pi}{2})$, we define the cone
\[
\Om_{\om,\th}
:=
\bigl\{x\in\RRR^3 \,:\, \angle(x,\om)\le \th \bigr\},
\]
that is, the cone in $\RRR^3$ with vertex at the origin, axis $\om$, and opening angle $\th$.
\end{df}
We can now state the main theorem, which constructs $(s,q)$--asymptotically flat initial data for any $s\in[1,3)$ and $q\in\mathbb{N}$ in the sense of Definition \ref{defasym}.
\begin{thm}\label{s=1initial}
Let $s\in[1,3)$, $q\in\mathbb{N}$, and let $\ep_0>0$ be sufficiently small. Then there exists a nontrivial $(s,q)$--asymptotically flat initial data set $(\RRR^3,g,k)$ with size $\ep_0$. More precisely, $(g,k)$ solves \eqref{MT1.1} and satisfies the following properties:
\begin{enumerate}
    \item The following pointwise decay bounds hold:
    \begin{align}\label{decay}
        g-\de
        =O_{q+1}\!\left(\frac{\ep_0}{\xja^{\frac{s-1}{2}}}\right),
        \qquad
        k
        =O_q\!\left(\frac{\ep_0}{\xja^{\frac{s+1}{2}}}\right).
    \end{align}
    \item The perturbation $(g-\de,k)$ is supported in a cone $\Om:=\Om_{\om,\th}$ with $\th\in(0,\frac{\pi}{2})$, and satisfies the bound
    \begin{equation}
        \|(g-\de,k)\|_{H_b^{q+1,\frac{s-4}{2}}(\Om)\times H_b^{q,\frac{s-2}{2}}(\Om)}
        \les\ep_0.
    \end{equation}
    \item For any $s'>s$, we have
    \begin{align}\label{noimprovements}
        (g-\de,k)\notin H_b^{q+1,\frac{s'-4}{2}}(\Om)\times H_b^{q,\frac{s'-2}{2}}(\Om).
    \end{align}
\end{enumerate}
\end{thm}
\begin{rk}
Taking $s=1$ and $q\ge 3$ in Theorem~\ref{s=1initial}, we obtain the corresponding borderline initial data
\[
(g,k)\in H_b^{q+1,-\frac{3}{2}}(\RRR^3)\times H_b^{q,-\frac{1}{2}}(\RRR^3).
\]
As an immediate consequence of the exterior stability result established in \cite{Shen24}, this initial data admits a future development in the exterior of an outgoing null cone. However, in view of \eqref{noimprovements}, the global stability result of \cite{Shen23} does not apply to this class of data. Consequently, the problem of global stability of Minkowski spacetime for the initial data $(g,k)$ constructed in Theorem \ref{s=1initial} remains open.
\end{rk}
\subsection{Outline of the proof}\label{subsec:outline-s1-vs-sp}
We sketch the main ideas of the proof of Theorem \ref{s=1initial}.
\begin{enumerate}
\item \textbf{Set-up.}
Let $\om\in\SSS^2$ and $\th\in(0,\frac{\pi}{2})$, and fix a cone $\Om=\Om_{\om,\th}\subset\RRR^3$, and denote $r:=|x|$.
\item \textbf{Cone-supported linear seed.} The role of the linear seed is to saturate the critical $s$--weight while remaining exactly in the kernel of the linearized constraint operator. Let $L$ be an admissible borderline decay weight function.\footnote{e.g. $L(r)=\log(2+r)$, see Proposition \ref{proph0pi0} for more details.} Choose scalar potentials $\phi_h,\phi_\pi$ and apply a cutoff so that $\supp(\phi_h,\phi_\pi)\subset\Om$. Fix a subcone $\Om_o\subset\subset\Om$ such that the cutoff equals 1 on $\Om_o\cap\{r\ge 2\}$. On $\Om_o\cap\{r\ge 2\}$, set
\[
\phi_h(x)=\ep_0\,\frac{r^{\frac{5-s}{2}}}{L(r)},
\qquad
\phi_\pi(x)=\ep_0\,\frac{r^{\frac{3-s}{2}}}{L(r)}.
\]
Define
\begin{equation}\label{eq:outline-seed}
h_0^{ij}:=\pr^i\pr^j\phi_h-\de^{ij}\De\phi_h,
\qquad
\pi_0^{ij}:=\pr^i\pr^j\phi_\pi-\de^{ij}\De\phi_\pi.
\end{equation}
Then $\supp(h_0,\pi_0)\subset\ov\Om$, and the seed tensors satisfy the linearized constraints
\[
\pr_i\pr_j h_0^{ij}=0,
\qquad
\pr_i\pi_0^{ij}=0.
\]
Moreover, for every $\ell\ge0$ and $x\in\Om_o\cap\{r\ge2\}$,
\[
|\pr^\ell h_0(x)|\simeq_\ell \ep_0\,\frac{r^{-\frac{s-1}{2}-\ell}}{L(r)}, \qquad |\pr^\ell \pi_0(x)|\simeq_\ell \ep_0\,\frac{r^{-\frac{s+1}{2}-\ell}}{L(r)}.
\]
The choice of $L$ ensures that $(h_0,\pi_0)$ belongs to the target space with weight $s$, but not to any strictly stronger weighted space with weight $s'>s$.
\item \textbf{Nonlinear correction via a conical right inverse.}
Write $(h,\pi)=(h_0+h_1,\pi_0+\pi_1)$ and $P(h,\pi)=\Phi(h,\pi)$. Let $S$ be the solution operator constructed in Proposition \ref{prop:S}, so that
\[
PS(f,F^j)=(f,F^j),
\qquad 
\supp(f,F^j)\subset\ov\Om \ \Longrightarrow\ \supp S(f,F^j)\subset\ov\Om.
\]
We solve the fixed point problem
\begin{equation}\label{eq:outline-fixed-point}
(h_1,\pi_1)
=S\Phi(h_0+h_1,\pi_0+\pi_1)\quad \text{ in } \quad
H_b^{q+1,\frac{s-4}{2}}(\Om)\times H_b^{q,\frac{s-2}{2}}(\Om),
\end{equation}
using the smoothness and Lipschitz properties of $\Phi$ from Proposition \ref{MT9}, together with the boundedness of $S$ from Proposition \ref{prop:S}. For $\ep_0$ sufficiently small, \eqref{eq:outline-fixed-point} admits a unique solution satisfying $\|(h_1,\pi_1)\|\lesssim \ep_0^2$. Consequently, $(h,\pi)$ solves \eqref{PPhi} and satisfies $\supp(h,\pi)\subset\ov\Om$.
\item \textbf{Size of the correction.}
Since $\Phi$ is quadratic, the source term is of size $O(\ep_0^2)$ in the same weighted space. By the mapping properties of $S$, this yields improved decay for the correction terms:
\begin{align*}
|\pr^\ell h_1(x)|\les_\ell \ep_0^2 r^{1-s-\ell},\qquad
|\pr^\ell \pi_1(x)|\les_\ell \ep_0^2 r^{-s-\ell},
\qquad r\to\infty.
\end{align*}
\item \textbf{Sharpness of the weight.}
By construction, for any $s'>s$ we have
\[
(h_0,\pi_0)\in
\Bigl(H_b^{q+1,\frac{s-4}{2}}(\Om)\times H_b^{q,\frac{s-2}{2}}(\Om)\Bigr)\setminus\Bigl(H_b^{q+1,\frac{s'-4}{2}}(\Om)\times H_b^{q,\frac{s'-2}{2}}(\Om)\Bigr).
\]
The non-improvement \eqref{noimprovements} follows by a contradiction argument using the mapping properties of $\Phi$ and the solution operator $S$: if $(h,\pi)$ had the stronger weight $s'$, then $\Phi(h,\pi)$ would lie in the corresponding $H_b^{q-1,\frac{s'}2}$ space, hence $(h_1,\pi_1):=S\Phi(h,\pi)$ would also have weight $s'$, forcing $(h_0,\pi_0)=(h,\pi)-(h_1,\pi_1)$ to have weight $s'$, contradicting the construction.
\end{enumerate}
\subsection{Acknowledgments}
The authors thank Xuantao Chen, Sergiu Klainerman, J\'er\'emie Szeftel and Arthur Touati for their interest and helpful discussions. The authors also thank Sung-Jin Oh for valuable discussions related to \cite{MaoTao}. J.W. is supported by ERC-2023 AdG 101141855 BlaHSt.
\section{Construction of the solution operator}
\subsection{Linearized problem}\label{ssec2.1}
We introduce new variables $(h,\pi)$ by
\begin{align}
\begin{split}\label{dfhpi}
h_{ij}&=g_{ij}-\de_{ij}-\de_{ij}\tr_\de(g-\de),\\
\pi_{ij}&=k_{ij}-\de_{ij}\tr_\de k.
\end{split}
\end{align}
The inverse transformation is given by
\begin{align}
\begin{split}\label{hpiinverse}
g_{ij}&=\de_{ij}+h_{ij}-\frac12 \tr_\de h\,\de_{ij},\\
k_{ij}&=\pi_{ij}-\frac12 \tr_\de \pi\,\de_{ij}.
\end{split}
\end{align}
We then define the operator
\begin{align*}
P(h,\pi):=(\pr_i\pr_j h^{ij},\, \pr_i\pi^{ij}).
\end{align*}
In terms of the variables $(h,\pi)$, the constraint equations \eqref{MT1.1} can be written as
\begin{align}\label{hpieq}
P(h,\pi)=\Phi(h,\pi),
\end{align}
where $\Phi(h,\pi):=(M(h,\pi),N^j(h,\pi))$ and
\begin{align*}
M(h,\pi) &= h\c \pr^2 h+\pr h\c\pr h+\pi\c\pi,\\
N^j(h,\pi) &= h\c\pr\pi+\pr h\c\pi.
\end{align*}
Here and in what follows, the notation $u\c v$ denotes a linear combination of contractions of $u$ and $v$. We are thus led to study the linearized system
\begin{align}
\begin{split}\label{prprh}
P(h,\pi)=0.
\end{split}
\end{align}
The core of the argument in \cite{MaoTao} consists of the construction of an explicit solution operator for \eqref{prprh}, see Proposition \ref{prop:S} for its explicit formula and properties.
\subsection{Fundamental kernels for the divergence equations}
We recall that the linear operator $P$ in \eqref{prprh} consists of the symmetric double-divergence and the symmetric divergence. The main input from \cite{MaoTao} is the existence of conical fundamental kernels for these operators, constructed by a ray method and then averaged over directions.
\begin{thm}[Conical fundamental kernels {\cite[Theorem 4]{MaoTao}}]\label{MT4}
    Let $v\in\SSS^2$ and $\chi\in C^\infty(\SSS^2)$ such that $\int_{\SSS}\chi=1$. There exists $K_\chi,L_{\chi,v}\in\DD'(\RRR^3)$ such that
    \begin{align}
        \pr_i\pr_jK^{ij}_\chi&=\de_0,\label{prprhde0}\\
        \pr_iL^{ij}_{\chi,v}&=\de_0 v^j.\label{prpide0}
    \end{align}
    Moreover, they satisfy the following properties:
    \begin{itemize}
        \item $K_\chi$ and $L_{\chi,v}$ are symmetric $2$--tensors;
        \item $K_\chi$ and $L_{\chi,v}$ are supported in the convex hull of the cone
        \begin{align*}
            \ov{\left\{x\in\RRR^3,\frac{x}{|x|}\in\supp\chi\right\}};
        \end{align*}
        \item $K_\chi$ is homogeneous of degree $-1$ while $L_{\chi,v}$ is homogeneous of degree $-2$.
        \item $K_\chi$ and $L_{\chi,v}$ are smooth in $\RRR^3\setminus \{0\}$.
    \end{itemize}
\end{thm}
\begin{rk}
The idea in \cite{MaoTao} is to solve by integrating along rays. Along a fixed ray $\{r\om:r>0\}$, the divergence and double divergence reduce to $\frac{d}{dr}$ and $\frac{d^2}{dr^2}$, so one recovers the unknowns by directly integrating in $r$. Averaging over $\om$ against $\chi$ (and symmetrizing if needed) produces the conical kernels $K_\chi$ and $L_{\chi,v}$. This conic solution-operator viewpoint for divergence-type equations was originated in \cite{OT}. More recently, \cite{IMOT} generalizes this idea to a general framework for underdetermined systems based on recovery on curves and finite-dimensional cokernel conditions.
\end{rk}
\section{Estimates on the \texorpdfstring{$b$}{}--Sobolev spaces}
\subsection{\texorpdfstring{$b$}{}--Sobolev spaces}
We recall basic estimates for $b$--Sobolev spaces in this section and prove the solution operators are bounded on $b$--Sobolev spaces.
\begin{df}\label{bSobolev}
    For $k\in\NNN$, the $b$--Sobolev space $H^k_b$ is defined by the norm:
    \begin{align*}
        \|u\|_{H_b^k(\RRR^3)}^2:=\sum_{i=0}^k\|\xja^i\pr^iu\|_{L^2(\RRR^3)}^2.
    \end{align*}
    We extend the definition to $k\in\RRR$ by duality and interpolation. We further define for $\de\in\RRR$
    \begin{align*}
        H_b^{k,\de}:=\xja^{-\de}H_b^k.
    \end{align*}
\end{df}
\begin{rk}
The $b$--Sobolev space captures the property that the decay rate of a function improves by $\xja^{-1}$ after taking a derivative, which is consistent with Definition \ref{defasym}. However, another interesting and difficult problem is to consider the initial data in standard Sobolev spaces $H^k(\RRR^3)$ but not in $H^k_b(\RRR^3)$.
\end{rk}
\begin{prop}\label{MT3}
    We have the following properties for any $\de\in\RRR$:
    \begin{enumerate}
        \item The following inequality holds for $k>\frac{3}{2}$:
        \begin{align}\label{embedding}
            \|\xja^{\frac{3}{2}+\de}u\|_{L^\infty}\les \|u\|_{H_b^{k,\de}}.
        \end{align}
        \item The following multiplication estimate holds:
        \begin{align}\label{bilinear}
            \|uv\|_{H_b^{k,\de}}\les \|u\|_{H_b^{k_1,\de_1}}\|v\|_{H_b^{k_2,\de_2}},
        \end{align}
        for $\de_1+\de_2=\de-\frac{3}{2}$, $k_1+k_2>0$,
        $$
        k=\Bigg\{\begin{aligned}
            \min(k_1,k_2),\qquad \max(k_1,k_2)>\frac{3}{2},\\
            k_1+k_2-\frac{3}{2},\qquad \max(k_1,k_2)\leq\frac{3}{2}.
            \end{aligned}
        $$
    \end{enumerate}
\end{prop}
\begin{proof}
Note that \eqref{embedding} is the standard Sobolev property. The proof of \eqref{bilinear} is based on the usual bilinear estimate (without weight), i.e. the Sobolev spaces as a Banach algebra. See {\cite[Proposition 2] {MaoTao}} for more details. 
\end{proof}
\begin{cor}\label{cor:MT3-endpoint}
Let $k>\frac32$ and $\de\ge-\frac32$. Then, we have
\begin{equation}\label{eq:MT3-mixed}
\|uv\|_{H_b^{k,\de}}\les \|u\|_{H_b^{k,\de}}\|v\|_{H_b^{k,\de}}.
\end{equation}
In particular, $H_b^{k,-\frac32}$ is a Banach algebra.
\end{cor}
\begin{proof}
The conclusion follows directly from \eqref{bilinear} choosing $(k_1,\de_1)=(k,\de)$ and $(k_2,\de_2)=\bigl(k,-\frac32\bigr)$, together with the assumption that $\de\geq -\frac32$.
\end{proof}
\subsection{Outgoing properties}
\begin{df}\label{outgoingdf}
    Let $\Om\subseteq\RRR^3$ and let $K(x,y)\in\DD'(\RRR^3\times\RRR^3)$ be the Schwartz kernel of an operator that we also denote by $K$. Suppose $K$ has the property
    \begin{align*}
        u\in C^\infty_c(\RRR^3),\qquad \supp u\subset\ov{\Om}\Longrightarrow\supp Ku \subset\ov{\Om}.
    \end{align*}
    We say that $K$ is \emph{outgoing} on $\Om$ if there exists $C=C(\Om)>1$ such that
    \begin{align*}
        K(x,y)=0,\quad \mbox{ on }\quad\big\{x,y\in\Om\big/\, |y|\geq C |x|\big\}.
    \end{align*}
\end{df}
\begin{prop}\label{ex4}
    Let $\om\in\SSS^2$, $\th\in(0,\frac{\pi}{2})$, $\Om=\Om_{\om,\th}$ and $\chi\in C^\infty_c(\Om\cap\SSS^2)$. The integration kernels $K(x,y)=K^{ij}_\chi(x-y), L(x,y)=L^{ij}_{\chi,e_k}(x-y)$ are outgoing on $\Om$, where $e_k$ denotes the unit vector in the $x_k$ direction.
\end{prop}
\begin{proof}
    As an immediate consequence of Theorem \ref{MT4}, $K^{ij}_\chi$ and $L^{ij}_{\chi,e_k}$ are supported in $\Om$. It therefore suffices to show that there exists a constant $C_\Om>0$, depending only on $\Om$, such that
\[
x,y\in\Om,\quad |y|\ge C_\Om |x|
\quad\Longrightarrow\quad
x-y\notin\Om .
\]
Let $\Om=\Om_{\om,\th}$ with $\th\in(0,\frac{\pi}{2})$. For $x,y\in\Om$, we have
\[
x\cdot\om \le |x|,\qquad y\cdot\om \ge |y|\cos\th .
\]
Hence
\[
(x-y)\cdot\om= x\cdot\om - y\cdot\om\leq |x|-|y|\cos\th=(|x|\sec\th -|y|)\cos\th.
\]
Choosing $C_\Om := \sec\th$, it follows that $(x-y)\cdot\om\leq0$ and hence $x-y\notin\Om$. We conclude that
\[
x,y\in\Om,\quad |y|\ge C_\Om |x|
\quad\Longrightarrow\quad
K^{ij}_\chi(x-y)=L^{ij}_{\chi,e_k}(x-y)=0 .
\]
This proves that the solution operator $S$ is outgoing on $\Om$ in the sense of Definition \ref{outgoingdf}.
\end{proof}
\begin{rk}
Proposition \ref{ex4} relies on the following geometric fact: if $x,y\in\Om_{\om,\th}$ with $\th\in(0,\frac{\pi}{2})$ and $|y|\ge C|x|$ for $C$ sufficiently large, then $x-y$ is dominated by $-y$ and therefore points outside the convex cone $\Om_{\om,\th}$.

By contrast, the weaker condition $|y|>|x|$ is not sufficient. To see this, take $\om=(0,0,1)$ and $\th\in(\frac{\pi}{3},\frac{\pi}{2})$, and choose an angle $\alpha\in (\pi-2\th,\th)$. Set $x=(0,0,1)$ and $y=r(\sin\alpha,0,\cos\alpha)$ with $r>1$. Then $x,y\in\Om_{\om,\th}$ and $|y|>|x|$. Moreover, defining
\[
F(r):=\frac{(x-y)\cdot\om}{|x-y|}
=\frac{1-r\cos\alpha}{\sqrt{1+r^2-2r\cos\alpha}},
\]
we compute
\[
F(1)=\sin\frac{\alpha}{2}>\cos\th.
\]
By continuity of $F(r)$, it follows that $F(r)\ge\cos\th$ for all $r\in(1,1+\ep]$, provided $\ep>0$ is sufficiently small. Consequently, $x-y\in\Om_{\om,\th}$ for such values of $r$.
\end{rk}
We record the following mapping property for outgoing homogeneous distributions.
\begin{prop}\label{MT7}
Let $\si\in(0,3)$ and let
\[
K\in C^\infty(\RRR^3\setminus\{0\})\cap\DD'(\RRR^3)
\]
be a homogeneous distribution of degree $\si-3$. Assume that the kernel $K(x-y)$ is outgoing. Then, for any $\de<\frac{3}{2}-\si$ and any $k>\frac{3}{2}$, the convolution operator
\[
u \longmapsto u*K
\]
extends to a bounded map
\[
H_b^{k-\si,\de+\si} \longrightarrow H_b^{k,\de}.
\]
\end{prop}
\begin{proof}
This result is derived in {\cite[Proposition 3] {MaoTao}}. In particular, in the proof of {\cite[Proposition 3] {MaoTao}}, the estimate for the outgoing component requires only the assumption $\de<\frac{3}{2}-\si$.
\end{proof}
\begin{lem}\label{MT8}
    Let $k>\frac{3}{2}$ and $\de\geq-\frac{3}{2}$. Then, in a small neighborhood of $\de_{ij}$, the inverse matrix map
    \begin{align*}
        \de_{ij}+H_b^{k,\de}&\to\de_{ij}+H_b^{k,\de},\\
        g_{ij}&\mapsto g^{ij}
    \end{align*}
    is smooth.
\end{lem}
\begin{proof}
The statement is proved in \cite[Corollary 5]{MaoTao} under the assumption $\de>-\frac32$. We briefly explain why the same argument extends to the endpoint case $\de=-\frac32$. The proof in \cite{MaoTao} is based on two ingredients. The first is the $L^\infty$ control provided by the Sobolev embedding at the critical weight $\de=-\frac32$, see \eqref{embedding}. The second is the boundedness of the multiplication in the weighted space $H_b^{k,\de}$, which holds for all $\de\geq -\frac32$ by \eqref{eq:MT3-mixed}. Since no additional decay is required in the argument, these two properties suffice to establish the smoothness of the inverse matrix map at the endpoint. Therefore, the conclusion of \cite[Corollary 5]{MaoTao} remains valid for $\de\geq -\frac32$.
\end{proof}
\begin{prop}\label{MT9}
For $q>\frac{1}{2}$ and $\de\geq -\frac{3}{2}$, the operator
\begin{align*}
\Phi:\; H_b^{q+1,\de}\times H_b^{q,\de+1} &\longrightarrow H_b^{q-1,\de+2},\\
( h,\pi )\ \qquad &\longmapsto \bigl(h\c \pr^2 h,\; \pr h\c \pr h,\; \pi\c \pi,\; h\c \pr\pi,\; \pr h\c \pi\bigr)
\end{align*}
is smooth in a small neighborhood of $(0,0)$.
\end{prop}
\begin{proof}
By Definition \ref{bSobolev}, the assumptions
\[
h\in H_b^{q+1,\de},\qquad\quad\pi\in H_b^{q,\de+1}
\]
imply
\[
\pr h\in H_b^{q,\de+1},\qquad\pr^2 h\in H_b^{q-1,\de+2},\qquad\pr\pi\in H_b^{q-1,\de+2}.
\]
Consequently, each component of $\Phi(h,\pi)$ is a product of two factors chosen from
\[
h,\quad \pr h,\quad \pr^2 h,\quad \pi,\quad \pr\pi.
\]
We estimate each such product using the bilinear estimate \eqref{bilinear}. In all cases, the output weight is given by
\[
(\de+1)+(\de+1)+\frac32\quad\text{ or }\quad\de+(\de+2)+\frac32,
\]
both of which are equal to $2\de+\frac72$. Since $\de\ge -\frac32$, we have
\[
2\de+\frac72 \geq \de+2,
\]
with equality at the endpoint $\de=-\frac32$.

For the Sobolev order, two cases arise. If one factor has the Sobolev order strictly greater than $\frac32$ (for example, $h\in H_b^{q+1,\de}$), then the $\min$--case of \eqref{bilinear} applies. Otherwise, \eqref{bilinear} yields the Sobolev order $2q-\frac32$, which is strictly greater than $q-1$ since $q>\frac12$. In either case, each component of $\Phi(h,\pi)$ belongs to $H_b^{q-1,\de+2}$ and continuously depends on $(h,\pi)$.

Finally, since the dependence of $\Phi(h,\pi)=(M(h,\pi),N^j(h,\pi))$ on $(h,\pi)$ involves only multiplication and the inverse matrix map, Proposition \ref{MT3} and Lemma \ref{MT8} implies that $\Phi$ is smooth in a neighborhood of $(0,0)$. This concludes the proof of Proposition \ref{MT9}.
\end{proof}
\section{Proof of the main theorem}
The proof has three components. In Section \ref{prop:S} we follow \cite{MaoTao} and define an outgoing right inverse $S$ for the linear system \eqref{prprh} using the conical kernels $K_\chi$ and $L_{\chi,v}$ from Theorem \ref{MT4}, and we show its boundedness on the relevant $H_b$ spaces. In Section \ref{sec4.1} we construct explicit seed solutions $(h_0,\pi_0)$ of \eqref{prprh} with prescribed borderline $H_b$ behavior. In Section \ref{s=1initial} we solve the full constraint equation \eqref{MT1.1} in the form \eqref{PPhi} by a contraction argument, using the fixed point formulation \eqref{fixedpointeq} and the bounds for $S$ and $\Phi$. Finally, we show that the borderline $H_b$ behavior persists despite the nonlinear correction.
\subsection{Solution operator for the divergence equations}
We first fix a conical region and package kernels from Theorem \ref{MT4} into a solution operator for \eqref{prprh}.
\begin{prop}\label{prop:S}
Let $\om\in\SSS^2$, $\th\in(0,\frac{\pi}{2})$, and fix a cone $\Om=\Om_{\om,\th}\subset\RRR^3$. Let $\chi\in C^\infty(\Om\cap\SSS^2)$ be a cutoff function, and let $K_\chi^{ij}$ and $L_{\chi,e_k}^{ij}$ be the distributions constructed in Theorem \ref{MT4}. For distributions $f,F^j\in\DD'(\RRR^3)$, we define the solution operator $S$ by
\begin{align*}
S(f,F^j):=\bigl(K_\chi^{ij} * f,\; L_{\chi,e_k}^{ij}*F^k\bigr),\qquad j=1,2,3.
\end{align*}
Then the following properties hold:
\begin{enumerate}
    \item The operator $S$ provides a right inverse of $P$ on the distributions supported in $\Om$, that is,
    \[
    PS(f,F^j)=(f,F^j)
    \]
    for all $(f,F^j)$ with $\supp(f,F^1,F^2,F^3)\subset\Om$.
    \item For any $s\ge 1$ and $q\in\NNN$, the operator $S$ extends to a bounded map
    \[
    S:\; H_b^{q-1,\frac{s}{2}}(\Om)
    \longrightarrow
    H_b^{q+1,\frac{s-4}{2}}(\Om)\times H_b^{q,\frac{s-2}{2}}(\Om).
    \]
\end{enumerate}
\end{prop}
\begin{proof}
From Theorem \ref{MT4}, we have
\begin{align*}
\pr_i\pr_j K_\chi^{ij} * f &= \de_0 * f = f,\\
\pr_i L_{\chi,e_k}^{ij} * F^k &= (e_k)^j\de_0*F^k = F^j,\qquad j=1,2,3.
\end{align*}
Thus, we obtain
\begin{align*}
PS(f,F^j) = (f,F^j),
\end{align*}
as claimed. Moreover, $\supp K_\chi,\supp L_{\chi,e_k}\subset\Om$ by Theorem \ref{MT4} and thus
\[
\supp(K_\chi^{ij}*f)\subset \supp K_\chi+\supp f\subset\Om, \qquad \supp(L_{\chi,e_k}^{ij}*F^k)\subset \supp L_{\chi,e_k}+\supp g\subset\Om,
\]
where we used the fact that $\Om$ is a cone, hence $\Om+\Om\subset\Om$. In particular, if $\supp(f,F^j)\subset\Om$, then 
\[\supp S(f,F^j)\subset\Om.\]
Together with Proposition \ref{ex4}, this shows that $S$ is an outgoing operator in the sense of Definition \ref{outgoingdf}. Finally, since $K_\chi^{ij}$ and $L_{\chi,e_k}^{ij}$ are homogeneous distributions of degree $-1$ and $-2$, respectively, we apply Proposition \ref{MT7} with $\de=\frac{s-4}{2}$. This implies that the convolution operator $f\mapsto K_\chi^{ij}*f$ is bounded from $H_b^{q-1,\frac{s}{2}}(\Om)$ to $H_b^{q+1,\frac{s-4}{2}}(\Om)$, while $F^j\mapsto L_{\chi,e_k}^{ij}*F^k$ is bounded from $H_b^{q-1,\frac{s}{2}}(\Om)$ to $H_b^{q,\frac{s-2}{2}}(\Om)$. Consequently, the operator $S$ extends to a bounded map
\[
H_b^{q-1,\frac{s}{2}}(\Om)\longrightarrow H_b^{q+1,\frac{s-4}{2}}(\Om)\times H_b^{q,\frac{s-2}{2}}(\Om).
\]
This concludes the proof of Proposition \ref{prop:S}.
\end{proof}
%============================================================
\subsection{Construction of solutions to the linearized equation}
\label{sec4.1}
%============================================================
We construct explicit cone-supported solutions to the linearized constraints \eqref{prprh} at the sharp $b$–Sobolev weights. The construction is formulated with a general \emph{borderline decay weight function} $L(r)$, assumed nondecreasing and subject to certain integrability and derivative bounds. One may keep in mind the model choice $L(r)=\log(2+r)$; see Remark~\ref{rk:examples-L} for further nontrivial examples.
\begin{prop}[Cone-supported seed with sharp $s$--weight]\label{proph0pi0}
Fix $\ep_0>0$, $s\in[1,3)$ and $q\in\NNN$. Let $L:[0,\infty)\to(0,\infty)$ be nondecreasing and assume there exists $R_1\ge1$ such that:
\begin{align}
\int_{R_1}^{\infty}\frac{dr}{r\,L(r)^2}<\infty, \qquad \int_{R_1}^{\infty}\frac{r^{\de-1}}{L(r)^2}\,dr=\infty\quad\forall\,\de>0,
\label{eq:L-conditions}
\end{align}
and there exists $C_k>0$ such that
\begin{align}
L\in C^{q+3}([R_1,\infty)),\qquad \big|(L^{-1})^{(k)}(r)\big|\le\frac{C_k}{r^k\,L(r)}\quad(0\le k\le q+3,\ r\ge R_1).
\label{eq:Linv-symbol}
\end{align}
Moreover, we assume that $L(r)$ satisfies
\begin{align}
\lim_{r\to\infty} r\,\frac{L'(r)}{L(r)}=0.
\label{eq:L-slowvar}
\end{align}
Then, there exists $(h_0,\pi_0)$ supported in $\ov\Om=\ov{\Om_{\om,\th}}$ with $\th\in (0,\frac{\pi}{2})$, solving \eqref{prprh} such that:
\begin{enumerate}
\item \emph{(Critical $s$ bound)} \quad
\begin{equation}\label{upperbound}
\|(h_0,\pi_0)\|_{H_b^{q+1,\frac{s-4}{2}}(\RRR^3)\times H_b^{q,\frac{s-2}{2}}(\RRR^3)}\les\ep_0.
\end{equation}
\item \emph{(Sharpness)} For any $s'>s$ and $q\in\NNN$,
\[
h_0\notin H_b^{q+1,\frac{s'-4}{2}}(\RRR^3),
\qquad
\pi_0\notin H_b^{q,\frac{s'-2}{2}}(\RRR^3).
\]
\end{enumerate}
\end{prop}
\begin{proof}
For the cone $\Om=\Om_{\om,\th}$ and choose $\chi\in C^\infty(\SSS^2)$, $\eta\in C^\infty([0,\infty))$ with
\[
\supp\chi\subset \Om\cap\SSS^2,\ \ \chi\equiv 1\text{ on }\Om_o\cap\SSS^2,\ \ \Om_o:=\Om_{\om,\th/2}, \qquad \eta=0\text{ on }[0,1],\ \eta=1\text{ on }[R_1,\infty).
\]
Let $r=|x|$, and define scalar functions
\begin{align*}
    \rho_h(r):=\frac{r^{\frac{5-s}{2}}}{L(r)},&\qquad \rho_\pi(r):=\frac{r^{\frac{3-s}{2}}}{L(r)}, \\ \phi_h(x):=\ep_0\,\eta(r)\rho_h(r)\chi\left(\frac{x}{|x|}\right),&\qquad \phi_\pi(x):=\ep_0\,\eta(r)\rho_\pi(r)\chi\left(\frac{x}{|x|}\right),
\end{align*}
The linear seed data $(h_0,\pi_0)$ then is constructed as
\begin{equation}\label{h0def}
h_0^{ij}:=\pr^i\pr^j\phi_h-\de^{ij}\De\phi_h, \qquad \pi_0^{ij}:=\pr^i\pr^j\phi_\pi-\de^{ij}\De\phi_\pi.
\end{equation}
Then it's straightforward to verify $\pr_i\pr_j h_0^{ij}=0$ and $\pr_i\pi_0^{ij}=0$, hence \eqref{prprh}, and $\supp(h_0,\pi_0)\subset\ov\Om$.
\paragraph{Upper bounds.}
On $r\ge R_1$ we have $\eta\equiv 1$, hence $\phi_h=\ep_0\rho_h(r)\chi\left(\frac{x}{|x|}\right)$ and $\phi_\pi=\ep_0\rho_\pi(r)\chi\left(\frac{x}{|x|}\right)$. We claim that for all integers $m\ge0$ and $r\ge R_1$,
\begin{equation}\label{eq:rho-deriv}
\big|\pr_r^m\rho_h(r)\big|\le C_m\,\frac{r^{\frac{5-s}{2}-m}}{L(r)}, \qquad \big|\pr_r^m\rho_\pi(r)\big|\le C_m\,\frac{r^{\frac{3-s}{2}-m}}{L(r)}.
\end{equation}
Indeed, $\rho_h(r)=r^{\frac{5-s}{2}}L(r)^{-1}$ and $\rho_\pi(r)=r^{\frac{3-s}{2}}L(r)^{-1}$; differentiating and applying Leibniz, every term is a product of a derivative of $r^{\gamma}$ (giving $r^{\gamma-\ell}$) and a derivative of $L^{-1}$. The bound \eqref{eq:Linv-symbol} gives $(L^{-1})^{(k)}(r)\lesssim r^{-k}L(r)^{-1}$, which yields \eqref{eq:rho-deriv}.

Next, angular derivatives satisfy $|\pr^\alpha(\chi\left(\frac{x}{|x|}\right))|\le C_\alpha r^{-|\alpha|}$. Combining this with \eqref{eq:rho-deriv} and Leibniz for $\pr^m(\rho(r)\chi\left(\frac{x}{|x|}\right))$, we obtain, for all integers $\ell\ge0$ and $r\ge R_1$,
\begin{equation}\label{eq:hpi-pointwise}
|\pr^\ell h_0(x)|\le C_\ell\,\ep_0\,\frac{r^{\frac{1-s}{2}-\ell}}{L(r)}, \qquad |\pr^\ell \pi_0(x)|\le C_\ell\,\ep_0\,\frac{r^{\frac{-1-s}{2}-\ell}}{L(r)}.
\end{equation}
Therefore, for $0\le \ell\le q+1$,
\[
\langle x\rangle^{\frac{s-4}{2}+\ell}|\pr^\ell h_0(x)| \les \ep_0\,\frac{r^{-\frac32}}{L(r)}\quad (r\ge R_1),
\]
and in polar coordinates on the fixed cone,
\[
\big\|\langle x\rangle^{\frac{s-4}{2}+\ell}\pr^\ell h_0\big\|_{L^2}^2 \les \ep_0^2\int_{R_1}^{\infty}\frac{dr}{r\,L(r)^2}\les\ep_0^2,
\]
where we used \eqref{eq:L-conditions} at the last step. Summing $\ell=0,\dots,q+1$ and adding the finite compact contribution from $r\le R_1$ gives $\|h_0\|_{H_b^{q+1,\frac{s-4}{2}}}\les \ep_0$, by \eqref{eq:L-conditions}. The same argument gives $\|\pi_0\|_{H_b^{q,\frac{s-2}{2}}}\les \ep_0$. This proves \eqref{upperbound}.
\paragraph{Lower bounds and $s'>s$ divergence.}
Fix $s'>s$ and set $\de:=s'-s>0$. On $\Om_o$ and $r\ge R_1$, $\chi\equiv 1$ and $\eta\equiv 1$, so $\phi_h,\phi_\pi$ are radial:
\[
\phi_h(r)=\ep_0\,\frac{r^{\frac{5-s}{2}}}{L(r)},\qquad \phi_\pi(r)=\ep_0\,\frac{r^{\frac{3-s}{2}}}{L(r)}.
\]
For radial $\phi$, the tensor $T^{ij}(\phi):=\pr^i\pr^j\phi-\de^{ij}\De\phi$ has a radial eigenvalue $-2\phi'(r)/r$, and hence $|T(\phi)(x)|\ge 2|\phi'(r)|/r$. Thus $|h_0|\ge 2|\phi_h'|/r$ and $|\pi_0|\ge 2|\phi_\pi'|/r$ on $\Om_o$.

Differentiate:
\[
\phi_h'(r)=\ep_0\,\frac{r^{\frac{5-s}{2}}}{L(r)}\Big(\frac{5-s}{2}\frac{1}{r}-\frac{L'(r)}{L(r)}\Big), \qquad \phi_\pi'(r)=\ep_0\,\frac{r^{\frac{3-s}{2}}}{L(r)}\Big(\frac{3-s}{2}\frac{1}{r}-\frac{L'(r)}{L(r)}\Big).
\]
By \eqref{eq:L-slowvar}, choose $R_2\ge R_1$ such that for all $r\ge R_2$, $r\frac{L'(r)}{L(r)}\le \frac12\min\{\frac{5-s}{2},\frac{3-s}{2}\}$. This yields
\[
|\phi_h'(r)|\gtrsim \ep_0\,\frac{r^{\frac{3-s}{2}}}{L(r)},\qquad |\phi_\pi'(r)|\gtrsim \ep_0\,\frac{r^{\frac{1-s}{2}}}{L(r)}\qquad (r\ge R_2),
\]
hence on $\Om_o\cap\{r\ge R_2\}$,
\[
|h_0(x)|\gtrsim \ep_0\,\frac{r^{\frac{1-s}{2}}}{L(r)},\qquad |\pi_0(x)|\gtrsim \ep_0\,\frac{r^{\frac{-1-s}{2}}}{L(r)}.
\]
Therefore, restricting the $L^2$ norms to $\Om_o\cap\{r\ge R_2\}$,
\[
\big\|\langle x\rangle^{\frac{s'-4}{2}}h_0\big\|_{L^2}^2 \ges \ep_0^2\int_{R_2}^{\infty}\frac{r^{\de-1}}{L(r)^2}\,dr=\infty, \qquad \big\|\langle x\rangle^{\frac{s'-2}{2}}\pi_0\big\|_{L^2}^2 \ges \ep_0^2\int_{R_2}^{\infty}\frac{r^{\de-1}}{L(r)^2}\,dr=\infty,
\]
by \eqref{eq:L-conditions}. This proves the sharpness for all $s'>s$.
\end{proof}
\begin{rk}[Examples of admissible borderline decay weight function $L$]
\label{rk:examples-L}
Define iterated logarithms by
\[
\log_{(1)}(t):=\log t,\qquad \log_{(j+1)}(t):=\log(\log_{(j)}(t)),\quad j\ge1.
\]
Fix $m\geq 1$ and let $1\leq j_0\leq m$. We introduce the exponents $(\b_1,\dots,\b_m)$ as follows:
\begin{align*}
\b_1=\b_2=\dots=\b_{j_0-1}=\frac{1}{2},\qquad \b_{j_0}>\frac{1}{2},\qquad \b_{j_0+1},\dots\b_m\in\RRR.
\end{align*}
Choose $R_*\geq 2$ so that
\begin{equation}\label{eq:log-iter-pos}
\log_{(j)}(2+r)\ge 2\qquad \forall\,1\leq j\leq m,\quad\forall\,r\geq R_\ast,
\end{equation}
and set for $r\ge R_\ast$,
\begin{equation}\label{eq:L-examples}
L(r):=\prod_{j=1}^{m}\Big(\log_{(j)}(2+r)\Big)^{\b_j}.
\end{equation}
Extend $L$ to $[0,R_\ast]$ as any positive nondecreasing $C^\infty$ function. It's easy to verify $L$ satisfies \eqref{eq:L-conditions}, \eqref{eq:Linv-symbol} and \eqref{eq:L-slowvar}. In particular, this includes $L(r)=(\log(2+r))^{\b_1}$ with $\b_1>\frac12$.
\end{rk}
\subsection{Proof of Theorem \ref{s=1initial}}
We denote
\begin{align*}
P(h,\pi)=(\pr_i\pr_jh^{ij},\pr_i\pi^{ij}),\quad\Phi(h,\pi)=\big(M(h,\pi),N^j(h,\pi)\big).
\end{align*}
Then, \eqref{MT1.1} becomes
\begin{align}\label{PPhi}
P(h,\pi)=\Phi(h,\pi).
\end{align}
Let $\Om=\Om_{\om,\th}$ be a cone and $(h_0,\pi_0)\in C^\infty(\RRR^3)$ be defined in Proposition \ref{proph0pi0}, which solves $P(h_0,\pi_0)=0$ and is supported in $\Om$. Let $S$ be the solution operator given by Proposition \ref{prop:S}. We consider the following fixed point problem
\begin{align}\label{fixedpointeq}
    (h_1,\pi_1)=S\Phi(h_0+h_1,\pi_0+\pi_1).
\end{align}
We define
\[
X:=\left\{(h,\pi)\Big/\,\|(h,\pi)\|_{H^{q+1,\frac{s-4}{2}}_b(\Om)\times H_b^{q,\frac{s-2}{2}}(\Om)}\leq\|(h_0,\pi_0)\|_{H^{q+1,\frac{s-4}{2}}_b(\Om)\times H_b^{q,\frac{s-2}{2}}(\Om)}\right\}.
\]
We have from Proposition \ref{proph0pi0}
\begin{align*}
\|(h_0,\pi_0)\|_{H^{q+1,\frac{s-4}{2}}_b(\Om)\times H_b^{q,\frac{s-2}{2}}(\Om)}\les\ep_0.
\end{align*}
According to Proposition \ref{MT9}, we have for $\ep_0$ small enough and $(h,\pi),(h',\pi')\in X$
\begin{align*}
\|\Phi(h_0+h,\pi_0+\pi)\|_{H_b^{q-1,\frac{s}{2}}(\Om)}&\les\ep_0^2,\\
\|\Phi(h_0+h,\pi_0+\pi)-\Phi(h_0+h',\pi_0+\pi')\|_{H_b^{q-1,\frac{s}{2}}(\Om)}&\les\ep_0\|(h,\pi)-(h',\pi')\|_X.
\end{align*}
By Proposition \ref{prop:S}, $S$ is bounded from $H_b^{q-1,\frac{s}{2}}(\Om)$ to $H_b^{q+1,\frac{s-4}{2}}(\Om)\times H_b^{q,\frac{s-2}{2}}(\Om)$. Thus, for $\ep_0$ small enough, the map
\begin{align*}
    F(h,\pi):=S\Phi(h_0+h,\pi_0+\pi):H^{q+1,\frac{s-4}{2}}_b(\Om)\times H_b^{q,\frac{s-2}{2}}(\Om)\to H^{q+1,\frac{s-4}{2}}_b(\Om)\times H_b^{q,\frac{s-2}{2}}(\Om)
\end{align*}
is a contraction on $X$. By the Banach fixed point theorem, there exists a unique fixed point $(h_1,\pi_1)\in X$ such that
\begin{align}
\begin{split}\label{h1pi1}
    S\Phi(h_0+h_1,\pi_0+\pi_1)&=(h_1,\pi_1),\\
    \|(h_1,\pi_1)\|_{H^{q+1,\frac{s-4}{2}}_b(\Om)\times H_b^{q,\frac{s-2}{2}}(\Om)}&\les\ep_0^2.
\end{split}
\end{align}
This implies that
\begin{align*}
    P(h_1,\pi_1)=PS\Phi(h_0+h_1,\pi_0+\pi_1)=\Phi(h_0+h_1,\pi_0+\pi_1).
\end{align*}
Denoting $h:=h_0+h_1$ and $\pi:=\pi_0+\pi_1$, we see that $(h,\pi)$ solves \eqref{hpieq} and satisfies
\begin{align*}
\|(h,\pi)\|_{H_b^{q+1,\frac{s-4}{2}}(\Om)\times H_b^{q,\frac{s-2}{2}}(\Om)}\les\ep_0.
\end{align*}
We now prove \eqref{noimprovements}. Let $s'>s$ and assume by contradiction that
\begin{align}\label{contradiction}
    (h,\pi)\in H_b^{q+1,\frac{s'-4}{2}}(\Om)\times H_b^{q,\frac{s'-2}{2}}(\Om).
\end{align}
Then, from Proposition \ref{MT9} we have $\Phi(h,\pi)\in H_b^{q-1,\frac{s'}{2}}$. Applying Proposition \ref{prop:S} and combining with \eqref{h1pi1}, we infer
\begin{align*}
    (h_1,\pi_1)=S\Phi(h,\pi)\in H_b^{q+1,\frac{s'-4}{2}}(\Om)\times H_b^{q,\frac{s'-2}{2}}(\Om).
\end{align*}
Recalling from Proposition \ref{proph0pi0} that $(h_0,\pi_0)\notin H_b^{q+1,\frac{s'-4}{2}}(\Om)\times H_b^{q,\frac{s'-2}{2}}(\Om)$, we deduce
\begin{align*}
    (h,\pi)=(h_0+h_1,\pi_0+\pi_1)\notin H_b^{q+1,\frac{s'-4}{2}}(\Om)\times H_b^{q,\frac{s'-2}{2}}(\Om),
\end{align*}
which contradicts \eqref{contradiction}. Combining with \eqref{hpiinverse}, we obtain \eqref{noimprovements}. This concludes the proof of Theorem \ref{s=1initial}.
\begin{rk}\label{rk:better-h1pi1}
In the proof of Theorem \ref{s=1initial}, the fixed point $(h_1,\pi_1)$ enjoys additional regularity beyond what is used above. More precisely:
 For $s\in[1,3)$ and $q\ge 4$,
    \[
    (h_1,\pi_1)\in H_b^{q+1,\frac{s-4}{2}}\times H_b^{q,\frac{s-2}{2}} \ \Longrightarrow\ (h_1,\pi_1)\in H_b^{\infty,\frac{s-4}{2}}\times H_b^{\infty,\frac{s-2}{2}},
    \]
    see {\cite[Proposition 7]{MaoTao}}.
\end{rk}
\begin{rk}\label{rk:better-h1pi1'}
In contrast to our linear seed data construction in Proposition \ref{proph0pi0}, the approach of \cite{MaoTao} toward Carlotto’s conjecture {\cite[Open Problem 3.18]{Car21}} is to build a compactly supported linear seed $(h_0,\pi_0)\in C_c^\infty(\RRR^3)$. Such choice allow decay improvement for $(h_1,\pi_1)$. In particular, if $s\in(1,3)$ and $q\geq 4$, then
\[
(h_1,\pi_1)\in H_b^{q+1,\frac{s-4}{2}}\times H_b^{q,\frac{s-2}{2}}
\quad \Longrightarrow \quad
(h_1,\pi_1)\in H_b^{\infty,-\frac12}\times H_b^{\infty,\frac12},
\]
see {\cite[Proposition 8]{MaoTao}} for more details. This follows from the fact that, for $s>1$, the nonlinear terms are of strictly lower homogeneity than the linear ones, allowing for an iterative regularity improvement. In contrast, this argument breaks down in the borderline case $s=1$, where the nonlinear terms have the same homogeneity as the linear terms.
\end{rk}

\end{document}